\documentclass{article}
\usepackage[utf8]{inputenc}
\usepackage{amsmath,amssymb,amsthm}
\usepackage{color}
\usepackage{graphicx}
\usepackage{hyperref}

\newcommand{\R}{\mathbb{R}}
\newcommand{\N}{\mathbb{N}}

\newcommand{\E}[1]{\mathbb{E}\left(#1\right)}

\newtheorem{theorem}{Theorem}
\newtheorem{lemma}[theorem]{Lemma}
\newtheorem{corollary}[theorem]{Corollary}
\newtheorem{remark}[theorem]{Remark}
\newtheorem{example}[theorem]{Example}

\DeclareMathOperator{\dist}{dist}
\DeclareMathOperator{\acosh}{acosh}

\title{Fold maps associated to geodesic random walks on non-positively curved manifolds}
\author{Pablo Lessa\thanks{IMERL, Facultad de Ingeniería, Universidad de la República, Montevideo, Uruguay \href{mailto:plessa@fing.edu.uy}{plessa@fing.edu.uy}} 
\and 
Lucas Oliveira\thanks{Department of Pure and Applied Mathematics, Universidade Federal do Rio Grande do Sul, Porto Alegre, Brazil \href{mailto:lucas.oliveira@ufrgs.br}{lucas.oliveira@ufrgs.br}}}

\begin{document}
\maketitle
\begin{abstract}
 We study a family of mappings from the powers of  the unit tangent sphere at a point to a complete Riemannian manifold with non-positive sectional curvature, whose behavior is related to the spherical mean operator and the geodesic random walks on the manifold.

 We show that for odd powers of the unit tangent sphere the mappings are fold maps.

 Some consequences on the regularity of the transition density of geodesic random walks, and on the eigenfunctions of the spherical mean operator are discussed and related to previous work.

 
 \medskip \noindent {\bf AMS2010:} Primary 57R45, 60J10, 53C22. 
\end{abstract}

\section{Introduction}

\subsection{Geodesic random walks and spherical mean operators}

An \(r\)-geodesic random walk on a complete Riemannian manifold \((M,g)\) is a Markov chain where at each step one picks a uniformly chosen direction and advances a (fixed) distance \(r\) along the geodesic in that direction.

\begin{figure}[h]
\begin{center}
\includegraphics[width=0.6\textwidth]{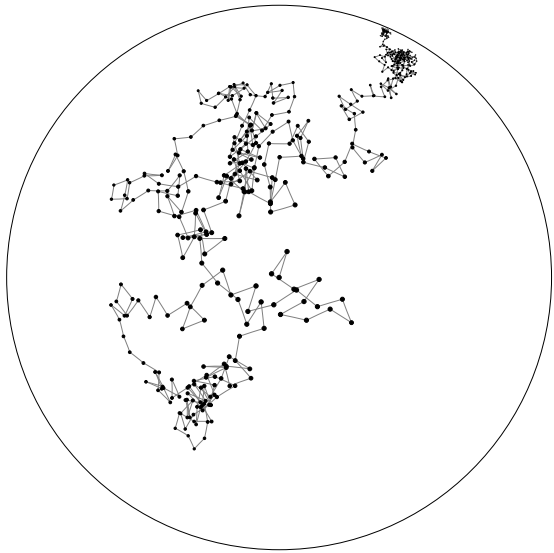}
\caption{An \(r\)-geodesic random walk in the Poincaré disk model of the hyperbolic plane.  Almost sure convergence to a boundary point was shown in \cite[Theorem 5.2]{furstenberg}.}
\end{center}
\end{figure}

If \(M\) is compact, then under very weak conditions on \(r\) (see \cite{sunada1978}, \cite{sunada1981}, \cite{sunada1983}), the volume measure is the unique stationary measure for the \(r\)-geodesic random walk.  
This suggests that performing a large number of steps of a geodesic random walk is a reasonable way of sampling a uniform point from the manifold (compare with \cite[Example 3C: How not to sample]{diaconis-shahshahani2013}).

Also, this motivates the issue of studying the ergodicity and mixing rates of geodesic random walks on compact Riemannian manifolds.

The Markov operator associated to the \(r\)-geodesic random walk is the \(r\)-spherical mean operator defined by
\[(L_rf)(x) = \int\limits_{T^1_xM}f(\exp_x(rv))dv\]
where \(T^1_xM\) is the unit tangent sphere at \(x\), \(\exp_x\) the exponential mapping, and integration is with respect to the rotationally invariant probability on the sphere.

The spherical mean operator has a self-adjoint continuous extension on \(L^2(M)\) with \(\|L_r\| \le 1\) (see \cite[Theorem A]{sunada1981}).

For small enough \(r\) the spherical mean operator is a Fourier integral operator of negative order (see \cite{tsujishita1976}).
This implies that if \(M\) is compact there is an orthonormal basis of \(L^2(M)\) consisting of eigenfunctions of \(L_r\), and all eigenfunctions of \(L_r\) are \(C^\infty\).

The conclusion was shown to hold for all \(r > 0\) such that the exponential map is an immersion on spheres of radius \(r\) by Sunada (see \cite[Theorem A]{sunada1981})

Regardless of whether \(M\) is compact or not, in some situations (e.g. \(M\) covers a compact manifold and \(r > 0\) is small enough) Tsujishita's result implies that, given \(k > 0\), the distribution of the \(n\)-th step of any geodesic random walk has a \(C^k\) density with respect to the volume measure for all \(n\) large enough.

In the case of compact symmetric spaces the spectrum of \(L_r\) has been determined explicitely (see \cite{pati-shahshahani-sitaram1995}).  

The mixing time for \(r\)-geodesic random walks on compact manifolds with strictly positive curvature has been studied in \cite{mangoubi-smith2018}.

In what follows we concentrate on the case where \((M,g)\) has non-positive sectional curvature.
Under this hypothesis, we will study a mapping \(\varphi: (T_x^1M)^n \to M\) first defined by Sunada (see \cite{sunada1983}) whose behavior is associated to that of the spherical mean operator and geodesic random walk.
We will show that \(\varphi\) has a nice structure from the point of view of singularity theory, and discuss some consequences for the spherical mean operator and geodesic random walk.

\subsection{Sunada's mapping \label{sunadamapping}}

On a complete Riemannian manifold \((M,g)\) fix a point \(x \in M\), a positive distance \(r > 0\), and a positive integer \(n \in \N\).  

Given \(v = (v_1,\ldots, v_n) \in (T_x^1M)^n\) we will define a piecewise geodesic \(\alpha_v\) parametrized by arc length and starting at \(x\).  The endpoint of \(\alpha_v\) will be the image of \(v\) by Sunada's mapping \(\varphi\).

To begin we set \(\alpha_v(t) = \exp_{x}(tv_1)\) for \(t \in [0,r]\) and let \(v_i(t)\) be the parallel transport of \(v_i = v_i(0)\) along \(\alpha_v\).   
Continuing inductively, for \(k = 1,\ldots,n-1\) and \(t \in [0,r]\) we define \(\alpha_v(kr+t) = \exp_{\alpha(kr)}(tv_{k+1}(kr))\) and extend the \(v_i(t)\) so they continue to be parallel along \(\alpha_v\).

With these definitions set \(\varphi(v) = \alpha_v(nr)\).

Notice that
\[(L_r^nf)(x) = \int\limits_{(T^1_xM)^n} f(\varphi(v_1,\ldots,v_n)) dv_1\cdots dv_n\]
where integration is with respect to the product \(m_n\) of \(n\) copies of the rotationaly invariant probability on \(T_x^1M\).

Also, the distribution of the \(n\)-th step of an \(r\)-geodesic random walk starting at \(x\) is the push-forward \(\varphi_* m_n\) of \(m_n\) under \(\varphi\).

\begin{figure}[h]
\begin{center}
\includegraphics[width=0.4\textwidth]{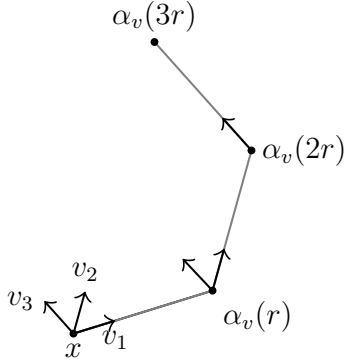}
\caption{The broken geodesic \(\alpha_v\) for \(v = (v_1,v_2,v_3)\).  Here \(n = 3\) and by definition \(\varphi(v) = \alpha_v(3r)\)}
\end{center}
\end{figure}

\subsection{Fold maps}

We recall the definition and basic properties of submersions with folds or fold maps, which are a generalization of Morse functions where the dimension of the codomain can be larger than one, see \cite[Chapter 3.4]{golubitsky-guillemin}.

Let \(f:X \to Y\) be a smooth mapping between smooth manifolds \(X\) and \(Y\).

Recall that \(f\) is a submersion if the tangent map \(D_xf: T_xX \to T_{f(x)}Y\) is surjective for all \(x \in X\) and an immersion if the tangent map is injective at all points.

We assume from now on that the dimension of \(X\) is greater than or equal to that of \(Y\).

We denote by \(S(f)\) the singular set of \(f\) i.e. the set at which \(D_xf\) is not surjective, and \(S_k(f)\) the set of points at which \(D_xf\) has corank \(k\).

Recall that \(f\) is a fold maps if \(S(f) = S_1(f)\), one has \(j^1f \pitchfork S_1\) (here \(j^1f\) is the first order jet of \(f\), and \(S_1\) is the set of co-rank one jets), and at each \(x \in S_1(f)\) one has \(T_{x}S_1(f) + \text{Ker}(D_xf) = T_xX\).

In the case when \(Y = \R\) fold maps are Morse functions.   

When \(f\) is a fold map the singular set is an embedded submanifold of \(X\) and the restriction of \(f\) to this submanifold is an immersion.

Conditions for the existence of a fold map between two manifolds have been studied for example in \cite{saeki}, \cite{ando}, and \cite{sadykov}.

\subsection{Statements}

\subsubsection{Singularities of Sunada's mapping and consequences}

We now state our main result.   Notice that for even \(n\) for all \(v \in T_x^1M\) and all choices of signs \(\sigma_i \in \lbrace -1,1\rbrace\) with \(\sum \sigma_i = 0\) one has \(\varphi(\sigma_1v,\ldots,\sigma_nv) = x\).   
This shows that \(\varphi\) is not an immersion on its singular set, so the second part of the conclusion below does not hold for even \(n\).

\begin{theorem}[Main theorem]\label{maintheorem}
Let \((M,g)\) be a complete Riemannian manifold with non-positive curvature.  Fix \(x \in M\), \(r > 0\), \(n \in \N\), and let \(\varphi\) be defined as above.

Then \(j^1\varphi \pitchfork S_1\) and, if \(n\) is odd, \(\varphi\) is a fold map.
\end{theorem}

\begin{remark}
 The statement above follows from the special case where \(M\) is simply connected by considering the universal covering space.
\end{remark}

The simplest case of the theorem above is the following:
\begin{example}
 Let \(S^{d-1} = \lbrace v \in \R^d: \|v\| = 1\rbrace\) be the unit sphere in \(\R^d\).    The mapping \(f:(S^{d-1})^n \mapsto \R^d\) given by 
 \[f(v_1,\ldots,v_n) = v_1+\cdots + v_n,\]
 is a fold mapping if \(n\) is odd.
\end{example}

As an immediate consequence we obtain (by a different proof, only in the case of non-positive curvature, for all \(r > 0\)) the following consequence of the results in \cite{tsujishita1976}.

\begin{corollary}\label{regularitycorollary}
In the context of theorem \ref{maintheorem}, if \(M\) is \(d\)-dimensional and \(n\) is odd, then the the distribution of the \(n\)-th step of a geodesic random walk on \((M,g)\) has \(C^k\)-density for \(k = (n-1)(d-1)/2 - 2\).
\end{corollary}
\begin{proof}
 Recall that the distribution in question is the push-forward under \(\varphi\) of the product measure on \((T_x^1M)^n\).

 By the local form of submersions with folds (see \cite[Theorem 4.5]{golubitsky-guillemin}), for each singular point there is a system of coordinates at the point and its image such that \(\varphi\) takes the form 
 \[(x_1,\ldots,x_{n(d-1)}) \mapsto (x_1,\ldots,x_{d-1}, \pm x_d^2 \pm x_{d+1}^2 \pm \cdots \pm x_{n(d-1)}^2)\]
 
 Hence the push-forward measure has the same regularity as the push-forward of measure in \(\R^{(n-1)(d-1)}\) with a smooth density to \(\R\) under the non-degenerate quadratic form in the last coordinate above.
 
 In particular, the regularity is that of the push-forward of a standard Gaussian measure in \(\R^{(n-1)(d-1)}\) under the quadratic form.
 Grouping the positive and negative terms one obtains that this is the distribution of \(Z = X-Y\) where \(X\) and \(Y\) are independent \(\chi^2\) random variables with parameters \(a\) and \(b\) (corresponding to the number of positive and negative signs in the quadratic form) respectively such that \(a+b = (n-1)(d-1)\).

 We now calculate the Fourier transform of the distribution and obtain
 \[\E{e^{itZ}} = \E{e^{itX}}\E{e^{-itY}} = (1-2it)^{-a/2}(1+2it)^{-b/2}.\]
 
 Since the Fourier transform decays like \(|t|^{-(a+b)/2} = |t|^{-(n-1)(d-1)/2}\) one obtains that the density is at least \(C^k\) where \(k = (n-1)(d-1)/2 - 2\)
\end{proof}

We have the following consequence for the \(r\)-geodesic random walk.
\begin{remark}\label{densitycorollary}
 In the context of theorem \ref{maintheorem}, let \(p_n(x,\cdot)\) be the probability density of the distribution of the \(n\)-th step of an \(r\)-geodesic random walk starting at \(x\).
 
 Then, \(p_n: M \times M \to [0,+\infty)\) has the same regularity as \(p_n(x,\cdot)\).
\end{remark}
\begin{proof}
 In a neighborhood \(U\) of any given \(x \in M\) the unit sphere bundle \(T^1M\) is \(C^\infty\) diffeomorphic to \(U \times S^{d-1}\) where \(S^{d-1}\) is the unit sphere in \(\R^d\).
 
 For fixed \(n\) and \(r > 0\), under this identification the mapping \((x,v) \mapsto \varphi_x(v)\) is smooth.

 This implies that \(p_n: M \times M \to [0,+\infty)\) has the same regularity as \(p_n(x,\cdot)\) as claimed.
\end{proof}

We also re-obtain, with a geometric proof, the following result from \cite[Theorem A]{sunada1981}.

\begin{corollary}
 In the context of theorem \ref{maintheorem}, if \((M,g)\) is compact then for all \(r\), \(L_r\) diagonalizable in \(L^2(M)\) and its eigenfunctions are \(C^\infty\).
\end{corollary}
\begin{proof} 
 Given \(k\), by remark \ref{densitycorollary} for \(n\) large enough \(L_r^n\) is continuous mapping from \(L^2(M)\) to \(C^k(M)\).
 
 If follows that \(L_r\) is compact and therefore diagonalizable.
 
 Furthermore, given any eigenfunction \(f\) with eigenvalue \(\lambda\) one has \(f = \lambda^{-n}L_r^nf\) for all \(n\) so \(f\) is \(C^\infty\). 
\end{proof}

\section{Structure of the singular set}

We first discuss the structure of the singular set of \(\varphi\).

\begin{lemma}\label{criticalsetlemma}
 In the context of theorem \ref{maintheorem} the following holds:  
 \begin{enumerate}
  \item At all critical points the tangent map has corank one, i.e. \(S(\varphi) = S_1(\varphi).\)
  \item The set of singular points is \(S(\varphi) = \lbrace (v_1,\ldots,v_n) \in (T_x^1M)^n: v_i = \pm v_1\text{ for all }i=2,\ldots,n\rbrace.\)
  \item If \((v_1,\ldots,v_n) \in S(\varphi)\) satisfies \(\sum\limits_{i = 0}^n v_i \neq 0\) then \(\varphi\) is an immersion when restricted to the the component of \(S(\varphi)\) containing \(v\).
 \end{enumerate}
\end{lemma}
\begin{proof}
 Given any \(v = (v_1,\ldots,v_n) \in (T_x^1M)^n\) define \(\alpha_v\) as in section \ref{sunadamapping}.
 During the proof we will consider curves \(w(s)\) with \(w(0) = v\), the associated the variation \(\gamma_s(t) = \alpha_{w(s)}(t)\), and corresponding Jacobi field \(J(t) = \partial_s \gamma_s(t)_{\vert s = 0}\). 
 
 To establish the first claim we consider variations for which  \(w'(0) = (0,\ldots,0,\dot{w})\) with \(\dot{w} \neq 0\).
 
 Notice that \(\gamma_s(t)\) restricted to \([a,b]\) with \(a = (n-1)r\) and \(b = nr\) is a geodesic variation.  Furthermore \(J(a) = 0\) and \(J'(a)\) is the parallel transport of \(\dot{w}\) along \(\alpha_v\) to \(\alpha_v(a)\).
 
 Since \(M\) has no conjugate points it follows that \(\dot{w} \mapsto J(b) = D_v\varphi v'(0)\) is injective.   This shows that the rank of \(D_v\varphi\) is at least \(d-1\) so \(S(\varphi) = S_1(\varphi)\) as claimed.
 
 Notice that we have also shown that the image of \(D_v\varphi\) always contains the space of vectors in \(T_{\varphi(v)}M\) perpendicular to \(\alpha'(nr)\). 
 
 To establish the second claim notice that for if \(a = kr\) and \(b = (k+1)r\) then, since \(\gamma_s\) is a geodesic of constant length on \([a,b]\), the first variation of energy applied to \(\gamma_s\) on \([a,b]\) yields
 \[g(J(a),\gamma_0'(a)) = g(J(b),\gamma_0'(b)).\]
 
 Assume there is \(k \in \lbrace 1,\ldots,n-1\rbrace\) such that \(v_k \neq \pm v_{k+1}\) and \(v_l = \pm v_n\) for all \(l \ge k+1\).
 
 Considering a curve with \(w'(0) = (\dot{w_1},\ldots,\dot{w_n})\) where \(\dot{w_i} \neq 0\) only for \(i = k\), notice that \(J((k+1)r)\) is non-zero and perpendicular to \(\alpha'((k+1)r^-)\).
 
 This implies that \(g(J((k+1)r),\alpha'((k+1)r^+)) \neq 0\) so by the first variation formula applied to the segment \([(k+1)r,nr]\) we have \(g(J(nr),\alpha'(nr)) \neq 0\).   
 Since \(J(nr) = D_v\varphi w'(0)\), and the image of \(D_v\varphi\) contains a vector not perpendicular to \(\alpha'(nr)\) and therefore \(D_v\varphi\) is surjective.
 
 It follows that if \(v \in S(\varphi)\) then \(v = (\pm v_0, \ldots, \pm v_0)\) as claimed.
 
 For the last claim we consider fixed signs \(\sigma_1,\ldots,\sigma_n \in \lbrace -1,+1\rbrace\), and a component of \(S(\varphi)\) of the form \(C = \lbrace v = (\sigma_1 v_0,\ldots,\sigma_n v_0): v_0 \in T_x^1M\rbrace\).
 
 Notice that, by hypothesis, \(R=  \left(\sum\limits_{i = 1}^n \sigma_i\right)r \neq 0\), and for \(v\) as above \(\varphi(v) = \exp_x(Rv_0)\).    Once again, since \(M\) has no conjugate points we obtain that \(\varphi\) restricted to \(C\) is an immersion on this set. 
\end{proof}

\section{Second order analysis near the singular set}

 Assume from now on, without loss of generality, that \(M\) is simply connected.
 
 Fix a vector \(v_0 \in T_x^1M\) and a choice of signs \(\sigma_1,\ldots,\sigma_n \in \lbrace -1, 1\rbrace\), and let \(v = (\sigma_1 v_0,\ldots, \sigma_n v_n)\).
 Notice that by lemma \ref{criticalsetlemma}, \(v\) is a critical point of \(\varphi\), and all critical points have this form.
 
 Set \(p = \exp(-2nrv_0)\), \(q = \exp(2nrv_0)\), and denote by \(f_p(x) = \dist(p,x)\) and \(f_q(x) = \dist(q,x)\) the distance functions to \(p\) and \(q\) respectively.
 Notice that \(f_p\) and \(f_q\) are smooth submersions when restricted to the image of \(\varphi\).
 
 Let \(V_p\) be the subset of \((T_x^1M)^n\) whose elements are those \(w = (w_1,\ldots,w_n)\) such that \(\alpha_{w}' = -\nabla f_p\) on \([kr,(k+1)r]\) for all \(k\) such that \(\sigma_{k+1} = -1\).

 Similarly let \(V_q\) be the set of elements satisfying \(\alpha_{w}' = -\nabla f_q\) on \([kr,(k+1)r]\) for all \(k\) such that \(\sigma_{k+1} = 1\).

 Notice that \(V_p\) and \(V_q\) are smooth submanifolds of \((T_x^1M)^n\) with complementary dimension, and which intersect transversally at \(v\).
 
 Concretely, an element in \(V_p\) is determined by the subset of coordinates corresponding to indices with \(\sigma_{i} = 1\) and elements of \(V_q\) are determined by the complementary set of coordinates.
 
  We endow \(V_p\) with the Riemannian metric coming from the coordinates corresponding to indices with \(\sigma_i = 1\).
  That is, a curve \(v(s) = (v_1(s),\ldots,v_n(s)) \in V_p\) is a geodesic if and only if each \(v_i(s)\) moves with constant speed along a great circle of \(T_x^1M\) for all \(i\) such that \(\sigma_i = 1\).

  Fix the Riemannian metric on \(V_q\) in the same manner using the complementary set of coordinates.

  \begin{figure}
 \includegraphics[width=\textwidth]{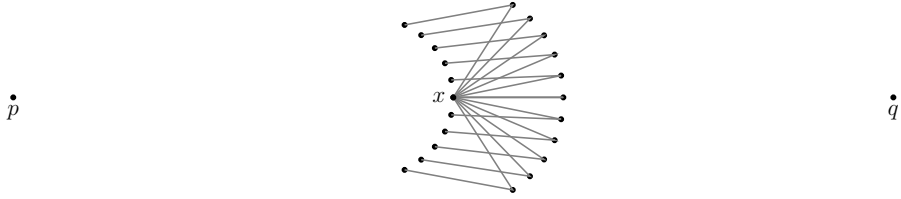}
\caption{A geodesic in in \(V_p\) in the Euclidean plane for \(n = 2\) and \(v = (v_0,-v_0)\).}
\end{figure}

\begin{figure}
 \includegraphics[width=\textwidth]{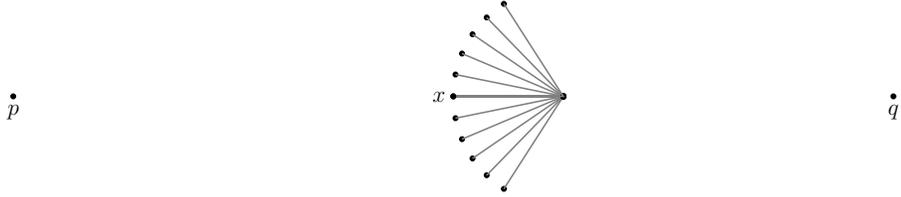}
\caption{A geodesic in \(V_q\) in the Euclidean plane for \(n = 2\) and \(v = (v_0,-v_0)\).}
\end{figure}

\begin{lemma}\label{acelerationlemma}
In the context above, if \(w(s)\) is a non-constant geodesic in \(V_p\) with \(w(0) = v\) then \(\partial_s^2 f_p(\alpha_{w(s)}(nr))_{\vert s = 0} < 0\).

Similarly, if \(w(s) = (w_1(s),\ldots,w_n(s))\) is a non-constant geodesic in \(V_q\) with \(w(0) = v\) then \(\partial_s^2 f_q(\alpha_{w(s)}(nr))_{\vert s = 0} < 0\).
\end{lemma}
\begin{proof}
We will give the argument when \(w(s)\) is a geodesic in \(V_p\), notice that changing \(v_0\) to \(-v_0\) this implies the statement for \(V_q\) as well.  To simplify notation let \(f = f_p\).

Let \(x_k(s) = \alpha_{v(s)}(kr)\) and \(c_k = \partial_s^2 f(x_k(s))_{\vert s = 0}\) for \(k = 0,\ldots,n\).

Notice that \(c_0 = 0\).  We claim that \(c_{k}\) is non-increasing with \(k\).

To see this notice that if \(\sigma_{k+1} = -1\) then for all \(s\) one has \(f(x_{k+1}(s)) = f(x_k(s)) - r\) from which \(c_{k+1} = c_k\).

On the other hand, if \(\sigma_{k+1} = 1\) then \(f(x_{k+1}(s)) - f(x_k(s))\) attains it maximum possible value \(r\) at \(s = 0\).
This implies by taking the second derivatives at \(s = 0\) that \(c_{k+1} - c_k \le 0\) which establishes the claim.

Now let \(k = 0,1,\ldots,n-1\) be minimal such that \(\sigma_{k+1} = 1\) and \(w_{k+1}'(0) \neq 0\).   We claim that \(c_{k+1} < 0\).

Notice that combined with the previous claim this implies that \(c_n < 0\) which would complete the proof.

To establish the claim, notice that, because \(w(s)\) is a geodesic in \(V_p\), one has that \(x_i(s)\) is constant for all \(i \le k\).  In particular \(x_k(s) = x_k\) is constant.

Let \(R = \dist(p,x_k)\) and \(B\) be the closed ball centered at \(p\) with radius \(R+r\).   
Notice that \(x_{k+1}(s)\), being at distance \(r\) from \(x_k\), is in \(B\) for all \(s\).

Let \(-a^2\) be a lower bound for sectional curvature on \(B\).

By Toponogrov's triangle comparison theorem, comparing with the simply connected space of constant curvature \(-a^2\) we have
\[f(x_{k+1}(s)) \le h(s) = \acosh\left(a^{-1}\cosh(ar)\cosh(aR) - a^{-1}\sinh(ar)\sinh(aR)\cos(\alpha(s))\right),\]
where  \(\alpha(s)\) is the angle at \(x_k\) of the triangle with vertices \(x_{k+1}(s),x_k,\gamma(T)\).

Notice that \(f(x_{k+1}(0)) = R+r = h(0)\) and \(\partial_sf(x_{k+1}(s))_{\vert s = 0} = h'(0) = 0\).

Also, because \(w(s)\) is a geodesic and parallel transport from \(x\) to \(T_{x_{k}}M\) along \(\alpha_v\) is an isometry, we have \(\alpha(s) = \pi + cs\) for some constant \(c \neq 0\).

It follows that \(h''(0) < 0\) which implies \(\partial_s^2 f(x_{k+1}(s))_{\vert s = 0} < 0\) as required.
\end{proof}

The result above may be interpreted qualitatively as follows.
\begin{remark}
 The statement above for \(V_p\) implies that \(\alpha_{w(s)}(nr)\) is strictly inside the ball centered at \(p\) with radius \(\dist(p,\alpha_{v}(nr))\) for for all non-zero \(s\) small enough.
\end{remark}

The following immediate corollary of the fact that spheres have positive curvature in \((M,g)\) will be used in the proof of theorem \ref{maintheorem}:
\begin{corollary}\label{accelerationcorollary}
In the context of lemma \ref{acelerationlemma} let \(\gamma(t) = \exp(tv_0)\) and \(A = \frac{D^2}{ds^2}\alpha_{w(s)}(nr)_{\vert s = 0}\) be the covariant acceleration of the endpoint \(\alpha_{w(s)}\).

Then \(g(A,\gamma') < 0\) if \(w(s)\) is a non-constant geodesic in \(V_p\) with \(w(0) = v\), and \(g(A,\gamma') > 0\) if \(w(s)\) is a non-constant geodesic in \(V_q\) with \(w(0) = v\).
\end{corollary}

\section{Proof of theorem \ref{maintheorem}}

We begin by deducing consequences from lemma \ref{criticalsetlemma}.

First, at all critical points of \(\varphi\) the image of the tangent map \(D\varphi\) has codimension one.

Furthermore, any critical point has the form \(v = (\sigma_1v_0,\ldots,\sigma_n v_0)\) for some \(v_0 \in T_x^1M\) and signs \(\sigma_1,\ldots,\sigma_d \in \lbrace -1,1\rbrace\).

Fix such a point \(v\) and consider the component \(S\) of the set of critical points containing \(v\) (which consists in maintaining the signs \(\sigma_i\) and varying the vector \(v_0\)).

Notice that \(S\) is an embedded sphere in \((T_x^1M)^n\) and has dimension \(d-1\).

If \(\sum v_i \neq 0\) one has that \(\text{ker}(D_v\varphi) \cap T_vS = \lbrace 0\rbrace\).  Notice that since, the dimension of the image of \(D_v\varphi\) is \(d-1\) this implies
\[\text{ker}(D_v\varphi) + T_vS = T_v(T_x^1M)^n.\]

In particular, this is satisfied at all critical points if \(n\) is odd.

This is the fold condition, except we haven't verified that the singularities of \(\varphi\) very the required transversality condition.

Hence, it remains to prove that \(j_1\varphi \pitchfork S_1\) to establish the theorem.

Let \(v,v_0,\sigma_1,\ldots,\sigma_n\) be as above but do not assume \(\sum v_i \neq 0\) or \(n\) odd.

Let \(p,q\) and \(V_p,V_q\) be defined as in lemma \ref{acelerationlemma}.  Let \(a\) and \(b\) be the dimension of \(V_p\) and that of \(V_q\) respectively.

We consider smooth local coordinates \((y_1,\ldots,y_{n(d-1)})\) identifying a neighborhood of \(v\) in \((T_x^1M)^n\) with a neighborhood of \(0\) in \(\R^a \times \R^b\) such that curves of the form
\[s \mapsto (sy_1,\ldots,sy_a,0,\ldots,0)\]
or 
\[s \mapsto (0,\ldots,0,sy_{a+1},\ldots,sy_{a+b})\]
correspond to geodesics through \(v\) in \(V_p\) and \(V_q\) respectively.

Using the exponential mapping at \(\varphi(v)\) we consider smooth local coordinates \((z_1,\ldots,z_d)\) identifying a neighborhood of \(\varphi(v)\) in \(M\) with a neighborhood of \(0\) in \(\R^d\) in such a way that the curves
\[t \mapsto (tz_1,\ldots,tz_d)\]
are geodesics with
\[t \mapsto (0,\ldots,0,t)\]
being a reparametrization of the geodesic \(\gamma(t)  = \exp_x(tv_0)\) and the subspace \(\lbrace z_d = 0\rbrace\) corresponding to the image under \(\exp_{\varphi(v)}\) of the subspace perpendicular to \(\gamma'\) at \(\varphi(v)\).

Set \(y = (y_1,\ldots,y_{n(d-1}))\), let \(\psi(y_1,\ldots,y_{n(d-1)}) = (\psi_1(y),\ldots,\psi_d(y))\) be the composition of \(\varphi\) with the corresponding coordinate changes, so that it maps a neighborhood of \(0\) in \(\R^{n(d-1)}\) to a neiborhood of \(0\) in \(\R^d\).

By lemma \ref{criticalsetlemma} the image of the tangent map \(D_0\psi\) is \(\R^{d-1} \times \lbrace 0\rbrace\).   It follows that \(\nabla\psi_d(0) = 0\) so \(0\) is a critical point of the real valued function \(\psi_d\).

On the other hand, since in normal coordinates the acceleration of a curve coincides with the usual second derivative at the instant it passes through \(0\) one has by corollary \ref{accelerationcorollary} that the second derivative of \(\psi_d\) at \(0\) along
any non-constant curve of the form
\[s \mapsto (sy_1,\ldots,sy_a,0,\ldots,0)\]
or 
\[s \mapsto (0,\ldots,0,sy_{a+1},\ldots,sy_{a+b})\]
is not zero.

This implies that \(\psi_d\) has a non-degenerate singularity at \(0 \in \R^{n(d-1)}\).

Therefore \(j_1\varphi \pitchfork S_1\) as required.

\section{Acknowledgments}

The author's would like to thank Gilles Courtois, Martin Reiris, and Rafael Ruggiero for their help. The second author is partially supported by CNPq grant 308489/2017-9.

\end{document}